\documentclass[11pt]{article}
\usepackage{graphicx} 
\usepackage{graphicx, amsfonts, amsmath, mathrsfs, amsthm, dsfont, mathtools, amssymb, tikz, tikz-cd, xparse, comment, titlesec, xcolor, pdfpages, adjustbox, hyperref, enumerate, bbm, datetime, datetime2, enumitem, aliascnt, cleveref, utfsym}
\usepackage{datetime2}
\hypersetup{
    colorlinks=true,
    linkcolor=blue,
    filecolor=magenta,      
    urlcolor=cyan,
}

\DTMlangsetup[en-GB]{ord=raise,monthyearsep={,\space}}

\newdateformat{monthyeardate}{%
  \monthname[\THEMONTH], \THEYEAR}

\usetikzlibrary{calc, decorations.pathmorphing, decorations.markings}

\definecolor{red}{rgb}{1,0,0}

\newtheorem{theorem}{Theorem}[section]

\newaliascnt{definition}{theorem}
\newtheorem{definition}[definition]{Definition}
\aliascntresetthe{definition}

\newaliascnt{conj}{theorem}
\newtheorem{conj}[conj]{Conjecture}
\aliascntresetthe{conj}

\newaliascnt{question}{theorem}

\aliascntresetthe{question}

\newaliascnt{claim}{theorem}
\newtheorem{claim}[claim]{Claim}
\aliascntresetthe{claim}

\newaliascnt{fact}{theorem}
\newtheorem{fact}[fact]{Fact}
\aliascntresetthe{fact}

\newaliascnt{note}{theorem}

\aliascntresetthe{note}

\newaliascnt{exercise}{theorem}

\aliascntresetthe{exercise}

\newaliascnt{lemma}{theorem}
\newtheorem{lemma}[lemma]{Lemma}
\aliascntresetthe{lemma}

\newaliascnt{prop}{theorem}
\newtheorem{prop}[prop]{Proposition}
\aliascntresetthe{prop}

\newaliascnt{corollary}{theorem}
\newtheorem{corollary}[corollary]{Corollary}
\aliascntresetthe{corollary}

\newaliascnt{obser}{theorem}

\aliascntresetthe{obser}

\newaliascnt{remark}{theorem}
\newtheorem{remark}[remark]{Remark}
\aliascntresetthe{remark}
\newtheorem*{theorem*}{Theorem}

\newcommand{\RR}{\mathbb{R}}

\newcommand{\flags}{\usym{1F3F3}}
\NewDocumentCommand\transnormdual{ggg}
{%
    \ensuremath{\IfNoValueTF{#1}%
        {\|J\|_{(AQ)^\circ \to AQ}}%
        {\IfNoValueTF{#3}%
            {\|J\|_{(#1)^\circ \to #2}}%
            {\|#1\|_{(#2)^\circ \to #3}}}%
    }%
}
\NewDocumentCommand\faces{ggg}
{%
    \ensuremath{\IfNoValueTF{#1}%
        {\mathcal{F}}%
        {\mathcal{F}_{#1}}%
    }%
}

\newcommand{\innerp}[2]{\left\langle #1,#2 \right\rangle}
\newcommand{\setdef}[2]{\left\{ #1 \ \middle| \ #2 \right\}}

\DeclareMathOperator{\Ima}{Im}

\DeclareMathOperator{\relint}{relint}
\DeclareMathOperator{\vol}{vol}
\DeclareMathOperator{\proj}{proj}
\DeclareMathOperator{\aff}{aff}
\DeclareMathOperator{\lin}{lin}
\DeclareMathOperator{\supp}{supp}

\DeclareMathOperator{\sign}{sign}

\DeclarePairedDelimiter{\parens}()

\usepackage{cleveref}

\crefname{conj}{conjecture}{conjectures}
\crefname{theorem}{theorem}{theorems}
\crefname{definition}{definition}{definitions}
\crefname{claim}{claim}{claims}
\crefname{question}{question}{questions}
\crefname{fact}{fact}{facts}
\crefname{note}{note}{notes}
\crefname{exercise}{exercise}{exercises}
\crefname{lemma}{lemma}{lemmata}
\crefname{prop}{proposition}{propositions}
\crefname{corollary}{corollary}{corollaries}
\crefname{obser}{observation}{observations}

\usepackage{appendix}
\usetikzlibrary{patterns, calc}

\usepackage[a4paper, left=1in, right=1in, top=1in, bottom=1in]{geometry}

\newaliascnt{fig}{theorem}

\crefname{fig}{figure}{figures}

\newaliascnt{sec}{theorem}
\crefname{sec}{section}{sections}

\newcommand{\graph}[1]{{\mathcal{G}_{#1}}}
\newcommand{\signedflags}[1]{{\flags^{#1}}}
\newcommand{\phist}{{\Phi_{st}}}
\newcommand{\calc}{\mathcal{C}}

\title{Kalai's flag conjecture for locally anti-blocking polytopes}
\author{Arnon Chor}
\date{\today}

\begin{document}

\maketitle
\begin{abstract}
We prove Kalai's full flag conjecture for the class of locally anti-blocking polytopes, and show that there is equality if and only if the polytope is a (generalized) Hanner polytope.
\end{abstract}

\section{Introduction and main theorem}

For polytopes, central symmetry plays a central role in the combinatorial structure. For example, the Figiel-Lindenstrauss-Milman inequality \cite{10.1007/BF02392234} tells us that a centrally symmetric polytope cannot have few facets and also few vertices. But what about faces of other dimensions? The question of counting all faces of centrally symmetric polytopes is the subject of a famous conjecture by Kalai \cite{kalai1989number}.

\begin{conj}[$3^d$ conjecture]
    Any centrally symmetric $d$-dimensional polytope has at least $3^d$ non-empty faces, with equality if and only if the polytope is a linear image of a Hanner polytope.
\end{conj}

Hanner polytopes, the conjectured equality cases, are recursively defined to be either the symmetric segment $[-1,1]$ in $\RR^1$, the polar dual $P^\circ = \setdef{x}{\forall y \in P: \innerp{x}{y} \leq 1}$ of some previously defined Hanner polytope $P$, or the convex hull $P \vee Q$ of two previously defined lower-dimensional Hanner polytopes embedded in some two orthogonal subspaces. Intriguingly, Hanner polytopes are also the conjectured minimizers for the famous Mahler conjecture \cite{mahler1939ubertragungsprinzip}:

\begin{conj}[Mahler's conjecture]
    For any centrally symmetric $d$-dimensional convex body $K$,
    \[
        \vol(K) \cdot \vol(K^\circ) \geq \frac{4^d}{d!} ,
    \]
    with equality if and only if $K$ is a linear image of a Hanner polytope.
\end{conj}

A (full) flag of a polytope $P$ is a chain of faces of $P$, one in each dimension. A later conjecture, also attributed to Kalai in \cite{schmitt2012ten} has to do with flags of centrally symmetric polytopes.

\begin{conj}[Kalai's (full) flag conjecture]
    Any centrally symmetric $d$-dimensional polytope has at least $2^d \cdot d!$ flags, with equality if and only if the polytope is a linear image of a Hanner polytope.
\end{conj}

Recall that a polytope $P$ is called \textit{locally anti-blocking} if for any $x \in P$ and coordinate subspace $H$ of $\RR^d$, $\proj_H P = P \cap H$ for $\proj_H$ the orthogonal projection on $H$. Note that a locally anti-blocking polytope need not be centrally symmetric. 
Locally anti-blocking polytopes and their various subclasses are natural families for testing conjectures: for example, Saint-Raymond \cite{saint1980volume} proved Mahler's conjecture for 1-unconditional polytopes (i.e. polytopes symmetric w.r.t. reflection about any coordinate hyperplane), and Artstein-Avidan, Sadovsky and Sanyal \cite{artstein2023geometric} generalized the result to locally anti-blocking polytopes; in the same paper \cite{artstein2023geometric} the authors also prove Godbersen's conjecture for the so-called anti-blocking bodies, later generalized by Sadovsky \cite{sadovsky2023godbersen} to locally anti-blocking bodies. The Hadwiger-Boltyanski illumination conjecture has been settled by Tikhomirov \cite{tikhomirov2017illumination} in another subclass, namely 1-symmetric polytopes; see also the works of Sun and Vritsiou \cite{sun2024illuminating,sun2024illumination}.

Recently, Sanyal and Winter, and independently Chambers and Portnoy \cite{sanyal2023kalai,chambers2022note}, proved the $3^d$ conjecture for locally anti-blocking polytopes; and Faifman, Vernicos and Walsh \cite{faifman2023volume} proved the flag conjecture for 1-unconditional polytopes (albeit without the equality case). The proof in \cite{faifman2023volume} uses highly non-elementary tools from Funk geometry. The main theorem of this paper is a strengthening of this result, and equally importantly, has an elementary inductive proof. Recall that generalized Hanner polytopes are defined similarly to Hanner polytopes, with the addition that in $\RR^1$ any segment with $0$ in its interior is also defined to be generalized Hanner. A locally anti-blocking polytope is called \textit{proper} if the origin is in its interior.

\begin{theorem}\label{thm:main}
    Any $d$-dimensional proper locally anti-blocking polytope has at least $2^d \cdot d!$ flags, with equality if and only if the polytope is a generalized Hanner polytope.
\end{theorem}

Let us make a simple reduction to an equivalent statement that is more convenient to work with.
A locally anti-blocking polytope $P$ is called \textit{normalized} if for any $e_i$ a standard basis vector of $\RR^d$, the support function $h_P(v) = \max_{x \in P} \innerp{x}{v}$ satisfies $h_P(\pm e_i) = 1$ (see \cite{sanyal2023kalai}). A normalized locally anti-blocking polytope is always contained in $\square^d$, the unit cube $[-1,1]^d$, and always contains $\lozenge^d$, its polar dual. Therefore, any normalized locally anti-blocking polytope is automatically proper. Note that a normalized generalized Hanner polytope is just a Hanner polytope. 

Given a locally anti-blocking polytope $P$, consider the map $\RR^d \to \RR^d$ given by
\[
    x \mapsto \sum_{i \in [d]: x_i \neq 0} \frac{x_i}{h_P(\sign x_i \cdot e_i)} e_i .
\]
The image of $P$ under this map is always a normalized locally anti-blocking polytope, and also has the same combinatorial structure as $P$ itself (see the notion of \textit{halfspace scaling} in \cite{sanyal2023kalai}). Thus, \Cref{thm:main} follows from the following statement, which is the version we proceed to prove.

\begin{theorem}\label{thm:main_equiv}
    Any $d$-dimensional normalized locally anti-blocking polytope has at least $2^d \cdot d!$ flags, with equality if and only if the polytope is a Hanner polytope.
\end{theorem}

\subsection{Acknowledgements}
The author would like to thank Professors Shiri Artstein-Avidan and Yaron Ostrover for their supervision and guidance. The author also thanks Boaz Slomka and Rei Henigman for listening to half-formed ideas, and Boaz Guberman for insightful talks about Hanner polytopes. The author was partially supported by the ERC under the European Union’s Horizon 2020 research and innovation programme (grant agreement no. 770127), by ISF grant Number 784/20, and by the Binational Science Foundation (grant no. 2020329).

\section{Notations and preliminaries}\label{sec:preq}

We denote $[d] = \{1,...,d\}$ and $[k,d] = \{k,...,d\}$ for $k \leq d$ integers, and the convex hull of a set $X \subseteq \RR^d$ by $\bigvee X$. The relative interior $\relint X$ of some set $X \subseteq \RR^d$ is defined as the interior of $X$ in the topology induced by $\aff X$, the affine span of $X$.

For a closed convex set $K \subseteq \RR^d$ and $\nu \in \RR^d$, let $K(\nu) = \setdef{x \in K}{\innerp{x}{\nu} = \sup_K \innerp{\cdot}{\nu}}$. A set of this form is called a face of $K$, and by convention the empty set $\emptyset$ is also a face of $K$. Denote by $\faces(K)$ the set of faces of $K$, including the empty face (which is of dimension -1 by convention) and $K$ itself. The faces $\emptyset$ and $K$ are called \textit{non-proper faces}. Denote by $\faces{i}(K)$ the set of faces of $K$ of dimension $i$.
A polyhedron is a closed convex set with a finite number of faces; equivalently, it is the intersection of a finite number of half-spaces. A polytope is a compact polyhedron. We call faces of a polyhedron of codimension 1 \textit{facets}, faces of dimension 1 \textit{edges}, and faces of dimension 0 \textit{vertices}. For a point $x \in P$, the support of $x$ in $P$, $\supp_P x$, is the unique face $F \in \faces(P)$ with $x \in \relint F$, or equivalently, the unique minimal face of $P$ that contains $x$. Similarly, for a set $A \subseteq P$ we denote by $\supp_P A$ the minimal face of $P$ that contains $A$. For a non-empty face $F \in \faces(P)$, denote by $N_P F = \setdef{\nu \in \RR^d}{P(\nu) \supseteq F}$ the \textit{normal cone} of $F$ in $P$.

Recall that a (polyhedral) cone is a set which is positively spanned by a finite number of vectors in $\RR^d$. A \textit{fan} in $\RR^d$ is a finite family $\Phi$ of polyhedral cones such that $\bigcup_{C \in \Phi} C = \RR^d$, and for any $C,D \in \Phi$, $C \cap D$ is a face of both $C$ and $D$, and is also a member of $\Phi$. We denote the set of $k$-dimensional cones of a fan $\Phi$ by $\Phi_k$.
For any fan $\Phi$ and point $p \in \RR^d$ denote by $\supp_\Phi p \in \Phi$ the minimal (w.r.t. inclusion) cone in $\Phi$ that contains $p$. For $J \subseteq [d]$ denote by $\RR^J$ the subspace of $\RR^d$ linearly spanned by the standard Euclidean basis vectors $\{e_j\}_{j \in J}$; a subspace of this form is called a \textit{coordinate subspace of $\RR^d$}.

A flag of a polyhedron $P$ of dimension $d$ is a collection $F = (F_{-1}, F_0, F_1, ..., F_d)$ with $F_i \in \faces{i}(P)$ for all $i \in [-1,d]$ and $F_i \subsetneq F_j$ for $i < j$. The set of flags of $P$ will be denoted $\flags(P)$. We note that this notion is sometimes referred to in the literature as a \textit{full} or \textit{complete} flag.

\subsection{Sketch of the proof and plan of the paper}

Let $P$ be a proper locally anti-blocking polytope and consider the standard fan $\phist$ in $\RR^d$, which consists of cones positively spanned by all subsets of $\{\pm e_i\}_{i=1}^d$ which don't contain both $\pm e_i$ for any $i$ (that is, $\phist$ is the minimal fan containing all the orthants in $\RR^d$). We start by introducing the notion of a \textit{sign} of a flag $F \in \flags(P)$ in \Cref{sec:signs}. The sign of $F$, denoted $\sign F$, is defined as the minimal cone in $\phist$ that intersects all the relative interiors of the faces composing $F$, and it is shown that $\dim \sign F \geq \dim P$. Therefore for the full-dimensionsional polytope $P$, the signs of all its flags are full dimensional cones in $\phist$.
For a cone $C \in \phist$, we denote by $\signedflags{C}(P)$ the set of flags of the intersection $P \cap \lin C$ that have sign exactly $C$, where $\lin C = \aff C - \aff C$ is the linear span. Note that these are not flags of $P$ itself, but of a section, see \Cref{fig:raise} for an example. In \Cref{sec:flip_lemma}, we state and prove \Cref{lem:unique_proj}, which will later be used in \Cref{sec:inequality}.

We count the flags of $P$ inductively as follows. For any cone $D \in \phist$ and one of its facets $C$ (which is also a cone of $\phist$ of course), we build an injection $\chi_C^D: \signedflags{C}(P) \to \signedflags{D}(P)$, and we show that these injections have disjoint images, for different facets $C$ of $D$. Since a cone $D \in \phist$ has exactly $\dim D$ facets, an inductive argument gives $|\signedflags{D}(P)| \geq \dim D !$ for all $D \in \phist$. Specifically for full-dimensional cones $D$, we get $|\signedflags{D}(P)| \geq d!$, and since there are $2^d$ full-dimensional cones (i.e. the orthants of $\RR^d$) and
\[
    \flags(P) = \bigsqcup_{\substack{D \in \phist \\ \dim D = d}} \signedflags{D}(P) ,
\]
we conclude that $|\flags(P)| \geq 2^d \cdot d!$ as needed.
This is carried out in \Cref{sec:inequality}.

For the equality case, we follow the strategy of Sanyal and Winter in \cite{sanyal2023kalai}: assume $P$ is a normalized locally anti-blocking polytope in $\RR^d$ and assume it has the correct number of flags. First, in \Cref{prop:min_downwards}, we show that for any coordinate hyperplane $H$, $P \cap H$ also has the correct number of flags. This allows us to define the graph $\graph{P}$, defined in \Cref{sec:properties}, which encodes the combinatorial structure of $P$ itself: in fact, $P$ can be recovered from $\graph{P}$ in this case, see \Cref{cor:graph_determines_P}. We use the classification in \Cref{cl:hanner_characterization} to assert that since all coordinate sections of $P$ are also minimizing, $P$ must be a Hanner polytope. This is done in \Cref{sec:equality}.

Let us give some more details.

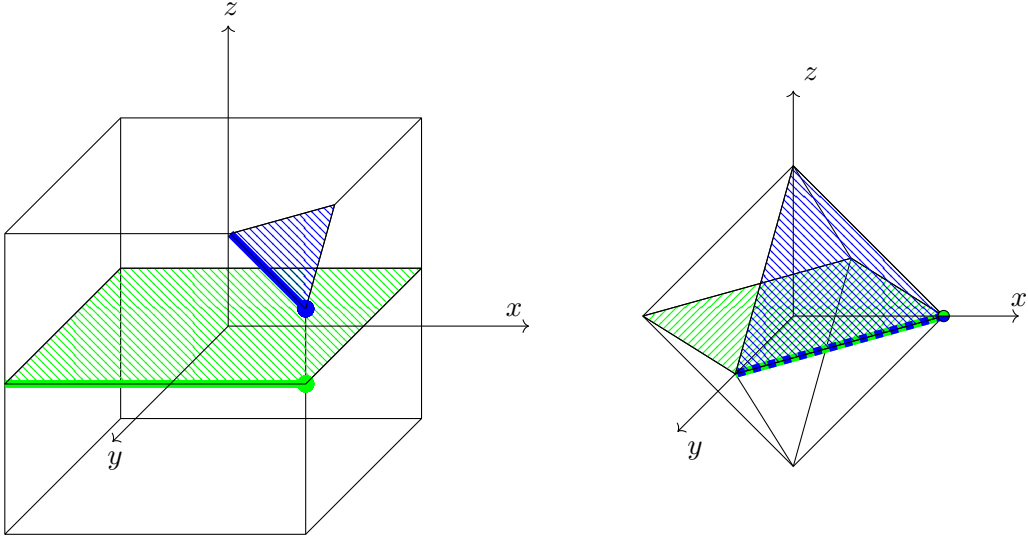
\begin{figure}
    \centering
    \begin{minipage}{.5\textwidth}
        \centering
        \begin{tikzpicture}[scale=2,line join=bevel]
            \coordinate (A1) at (1,-1,1);
            \coordinate (A2) at (1,-1,-1);
            \coordinate (A3) at (-1,-1,-1);
            \coordinate (A4) at (-1,-1,1);
            \coordinate (B1) at (1,0,1);
            \coordinate (B2) at (1,0,-1);
            \coordinate (B3) at (-1,0,-1);
            \coordinate (B4) at (-1,0,1);
            \coordinate (C1) at (1,0.5,1);
            \coordinate (C2) at (1,1,0.5);
            \coordinate (C3) at (0.5,1,1);
            \coordinate (D2) at (1,1,-1);
            \coordinate (D3) at (-1,1,-1);
            \coordinate (D4) at (-1,1,1);
            \coordinate (E1) at (0.1,0,0);
            \coordinate (E2) at (0,2,0);
            \coordinate (E3) at (2,0,0);
            \coordinate (E4) at (0,0,2);
            
            \draw (A1) -- (A2) -- (A3) -- (A4) -- cycle;
            \draw (B1) -- (B2) -- (B3) -- (B4) -- cycle;
            \draw (A1) -- (C1);
            \draw (A2) -- (D2);
            \draw (A3) -- (D3);
            \draw (A4) -- (D4);
            \draw (C1) -- (C2) -- (C3) -- cycle;
            \draw (C2) -- (D2) -- (D3) -- (D4) -- (C3);

            \draw[green] plot[mark=*, mark size=1.5] (B1);
            \draw[green,line width=3] (B1) -- (B4);
            \draw[pattern=north west lines, pattern color=green] (B1) -- (B4) -- (B3) -- (B2) -- cycle;

            \draw[blue] plot[mark=*, mark size=1.5] (C1);
            \draw[blue,line width=3] (C1) -- (C3);
            \draw[pattern=north west lines, pattern color=blue] (C1) -- (C2) -- (C3) -- cycle;

            \draw[->] (E1) -- ++(E2);
            \draw[->] (E1) -- ++(E3);
            \draw[->] (E1) -- ++(E4);
            \node[anchor=south] at (E3) {$x$};
            \node[anchor=north west] at (E4) {$y$};
            \node[anchor=south west] at (E2) {$z$};
        \end{tikzpicture}
    \end{minipage}%
    \begin{minipage}{.5\textwidth}
        \centering
        \begin{tikzpicture}[scale=2,line join=bevel]
            \coordinate (A1) at (1,0,0);
            \coordinate (A2) at (-1,0,0);
            \coordinate (B1) at (0,1,0);
            \coordinate (B2) at (0,-1,0);
            \coordinate (C1) at (0,0,1);
            \coordinate (C2) at (0,0,-1);
            
            \coordinate (E1) at (0,0,0);
            \coordinate (E2) at (0,1.5,0);
            \coordinate (E3) at (1.5,0,0);
            \coordinate (E4) at (0,0,2);
            
            \draw (A1) -- (B1) -- (A2) -- (B2) -- cycle;
            \draw (A1) -- (C1) -- (A2) -- (C2) -- cycle;
            \draw (B1) -- (C1) -- (B2) -- (C2) -- cycle;

            \draw[green,line width=3] (A1) -- (C1);
            \draw[pattern=north east lines, pattern color=green] (A1) -- (C1) -- (A2) -- (C2) -- cycle;
            
            \draw[blue,line width=3,dashed] (A1) -- (C1);
            \draw[pattern=north west lines, pattern color=blue] (A1) -- (C1) -- (B1) -- cycle;

            \node at (A1) [draw, circle, minimum size=0.15cm, inner sep=0pt, path picture={
                \clip (path picture bounding box.center) circle (0.15cm);
                \fill[green] (-0.2,0) rectangle (0.2,0.2);
                \fill[blue]  (-0.2,-0.2) rectangle (0.2,0);
            }] {};
            
            \draw[->] (E1) -- ++(E2);
            \draw[->] (E1) -- ++(E3);
            \draw[->] (E1) -- ++(E4);
            \node[anchor=south] at (E3) {$x$};
            \node[anchor=north west] at (E4) {$y$};
            \node[anchor=south west] at (E2) {$z$};
        \end{tikzpicture}
    \end{minipage}
    \caption{Two examples of some flag $F$ (in green) and $\chi_C^D F$ (in blue), where $C$ is the cone spanned by $x,y$ and $D$ is spanned by $x,y,z$.}
    \label{fig:raise}
\end{figure}

\subsection{Flags of polytopes and their signs}\label{sec:signs}

In this subsection, fix some fan $\Phi$. We note that in the rest of the paper, we will only use the standard fan $\phist$, nevertheless the definition makes sense for any fan. Given a flag $F \in \flags(P)$, we first show that there is a cone in $\Phi$ that intersects the relative interiors of all the faces $F_i$, and then show that the set of cones that have this property is closed under intersections. The intersection of all these cones will be denoted $\sign_\Phi F$, the \textit{sign} of $F$. 

We start with a simple fact about simplices. 

\begin{fact}\label{fact:face_directions}
    Let $S$ be a $k$-dimensional simplex in $\RR^d$ and let $x,y \in \RR^d$ be points such that $H = \aff (S \vee \{x,y\})$ is $(k+1)$-dimensional, and that $x,y$ lie in the same connected component of $H \setminus \aff S$. Then $\relint \parens*{x \vee S} \cap \relint \parens*{y \vee S} \neq \emptyset$.
\end{fact}

\begin{claim}\label{cl:sign}
    Let $P$ be a $k$-dimensional polytope in $\RR^d$ and $F = (F_{-1},F_0,...,F_k) \in \flags(P)$. Then there is a cone $C \in \Phi$ such that $C \cap \relint F_i \neq \emptyset$ for all $i \in [0,k]$. Additionally, for any cone $C \in \Phi$ which satisfies this condition, $\dim C \geq k$.
\end{claim}

\begin{proof}
    We prove the following stronger statement: for any $k \in [0,d]$, $k$-dimensional polytope $P$ and a flag $F \in \flags(P)$, there is a nested sequence of cones $C_0 \subseteq C_1 \subseteq ... \subseteq C_k$ in $\Phi$ such that for any $i \in [0,k]$, $\relint C_i \cap \relint F_i \neq \emptyset$ (the cones $C_i$ do not have to be strictly nested). This is indeed stronger than the original statement: the cone $C_k$ satisfies it. The proof is by induction on $k$. If $k=0$, note that $\supp_\Phi p \in \Phi$ satisfies the statement, for $p$ the unique point in $P$.
    
    Assume the statement holds for any $(k-1)$-dimensional polytope and its flags, we will show it holds for $P$, a $k$-dimensional polytope, and a given flag $F = (F_{-1},F_0,...,F_k) \in \flags(P)$. Since $F' = (F_{-1},F_0,...,F_{k-1})$ is a flag of $F_{k-1}$, which is $(k-1)$-dimensional, by induction there is a nested sequence of cones $C_0 \subseteq ... \subseteq C_{k-1}$ as above. Pick some $x \in \relint C_{k-1} \cap \relint F_{k-1}$.
    
    As $x \in \relint C_{k-1}$, there is some $r>0$ such that
    \begin{equation}\label{eq:1}
        x + rB \subseteq \bigcup_{C_{k-1} \subseteq D \in \Phi} D ,
    \end{equation}
    where $B$ is the unit Euclidean ball.
    Since $x \in F_{k-1}$ is in the relative boundary of $F_k$, there is some $y \in \relint F_k \cap (x + rB)$.
    Consider $C_k = \supp_\Phi y$. By \eqref{eq:1}, $C_k \supseteq C_{k-1}$, and by construction, $y \in \relint C_k \cap \relint F_k$.
    
    Additionally, let $C \in \Phi$ be a cone that has $C \cap \relint F_i \neq \emptyset$ for all $i \in [0,k]$. Choose points $x_i \in C \cap \relint F_i$. As the points $x_i$ lie in the relative interiors of their corresponding faces $F_i$, they must be in general position, and so do not all belong to any $(k-1)$-dimensional affine subspace. Since they all belong to $C$, it must hold that $\dim C \geq k$.
\end{proof}

\begin{claim}\label{cl:meet_closed}
    Let $P$ be a $k$-dimensional polytope in $\RR^d$ and let $F \in \flags(P)$. The set of cones $C \in \Phi$ for which $C \cap \relint F_i \neq \emptyset$ for all $i \in [0,k]$ is closed under intersection. That is, if the above condition holds for $C,D \in \Phi$, then it also holds for $C \cap D$. 
\end{claim}

\begin{proof}
    Let $C,D \in \Phi$ and assume that for any $i \in [0,k]$, $\relint F_i \cap C$ and $\relint F_i \cap D$ are nonempty. We will show that for any $i \in [0,k]$, $\relint F_i \cap C \cap D \neq \emptyset$ by induction on $i$.

    For $i = 0$, there is a unique point $p \in \relint F_0$ and by hypothesis $p \in C \cap D$.
    
    Assume the statement is true for any $j \in [0,i-1]$, we will show it holds for $i$ as well. By the induction hypothesis, there are points $x_j \in \relint F_j \cap C \cap D$, for $j \in [0,i-1]$. 
    Choose some $x \in \relint F_i \cap C$ and $y \in \relint F_i \cap D$.
    We first note that $\{x_j\}_{j=0}^{i-1} \cup \{x,y\}$ all belong to an $i$-dimensional affine subspace, namely $\aff F_i$, and that $\{x_j\}_{j=0}^{i-1}$ affinely span $\aff F_{i-1}$. Next, note that $x$ and $y$ lie on the same connected component of $\aff F_i \setminus \aff F_{i-1}$, since $F_{i-1}$ is a facet of $F_i$. 
    By \Cref{fact:face_directions}, there is some $z \in \relint \parens*{x \vee \bigvee_{j=0}^{i-1} x_j} \cap \relint \parens*{y \vee \bigvee_{j=0}^{i-1} x_j} \subseteq C \cap D$. Since $x \vee \bigvee_{j=0}^{i-1} x_j$ and $y \vee \bigvee_{j=0}^{i-1} x_j$ both affinely span $F_i$ and are contained in it, we see that $z \in \relint F_i$. See \Cref{fig:book}.
\end{proof}

\begin{figure}
    \centering
    \begin{tikzpicture}[scale=5,line join=bevel]
        \pgfsetxvec{\pgfpoint{1cm}{0cm}}
        \pgfsetyvec{\pgfpoint{-.5cm}{-.5cm}}
        \pgfsetzvec{\pgfpoint{0cm}{0.8cm}}
        
        \coordinate (S0) at (0,0,0);
        \coordinate (S1) at (0,0,1);
        \coordinate (CD0) at (1,1,0);
        \coordinate (CD1) at (1,1,1);
        \coordinate (C0) at (0,0.7,0);
        \coordinate (C1) at (0,0.7,1);
        \coordinate (D0) at (0.7,0,0);
        \coordinate (D1) at (0.7,0,1);
        \coordinate (x0) at (0.5,0.5,0.25);
        \coordinate (F1) at (-0.1,-0.1,0.25);
        \coordinate (x1) at ($(x0)!0.5!(F1)$);
        \coordinate (F2) at (0.85,0.85,1.1);
        \coordinate (x) at (0.5,0.5,0.75);
        \coordinate (y) at (0.1,0.1,0.75);

        \draw (S0) -- (S1) -- (CD1) -- (CD0) -- cycle;
        \draw (S0) -- (S1) -- (C1) -- (C0) -- cycle;
        \draw (S0) -- (S1) -- (D1) -- (D0) -- cycle;
        \draw[line width=0.5mm] (x0) -- (F1);
        \draw[line width=0.5mm] (x0) -- (F2);

        \node[anchor=north] at (x0) {$x_0$};
        \node[circle,draw=purple, fill=purple, inner sep=0pt, minimum size=3pt] () at (x0) {};
        \node[anchor=north] at (x1) {$x_1$};
        \node[circle ,draw=purple, fill=purple, inner sep=0pt, minimum size=3pt] () at (x1) {};
        \node[anchor=south] at (x) {$x$};
        \node[circle,draw=red, fill=red, inner sep=0pt, minimum size=3pt] () at (x) {};
        \node[anchor=south] at (y) {$y$};
        \node[circle,draw=blue, fill=blue, inner sep=0pt, minimum size=3pt] () at (y) {};

        \fill[fill=red, opacity=0.5] (x0) -- (x) -- (x1) -- cycle;
        \fill[fill=blue, opacity=0.5] (x0) -- (y) -- (x1) -- cycle;

        \node[anchor=south west] at ($(C0)!0.5!(CD0)$) {\large $C$};
        \node[anchor=south west] at ($(D0)!0.5!(CD0)$) {\large $D$};
    \end{tikzpicture}
    \caption{The cones $C,D$ intersecting at a plane; a point $x \in \relint F_2 \cap C$ (in red); a point $y \in \relint F_2 \cap D$ (in blue); and points $(x_j)_{j=0}^{1}$ (in purple).}
    \label{fig:book}
\end{figure}
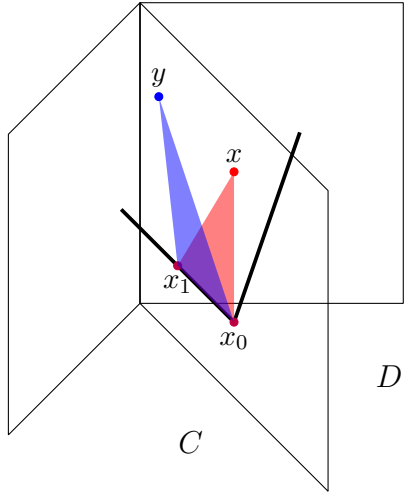

We are ready to define our flag signs.

\begin{definition}\label{def:sign}
    Let $P$ be a $k$-dimensional polytope in $\RR^d$ and let $F \in \flags(P)$. By \Cref{cl:sign} the set of cones $C \in \Phi$ such that $\relint F_i \cap C \neq \emptyset$ for all $i \in [0,k]$ is not empty, and by \Cref{cl:meet_closed} it is closed under intersections. Denote by $\sign_\Phi(F)$ the intersection of this set. This cone will be called the sign of the flag $F$. It follows from \Cref{cl:sign,cl:meet_closed} that $\dim \sign_\Phi F \geq k$.
    
    Denote by $\signedflags{C}(P)$ the set of flags of $P \cap \lin C$ with sign $C$. 
\end{definition}

\subsection{A flip lemma}\label{sec:flip_lemma}

In this subsection we prove \Cref{lem:unique_proj}, which will be of use later. Essentially, this lemma tells us that given any polytope $P$ and a flag $F \in \flags(P)$, there is a unique edge containing $F_0$ such that projecting $F$ along this edge results in a flag. To prove this lemma, we recall the definition of a \textit{flip}.

Recall that for any polytope $P$, $i \in [0,\dim P-1]$, and two faces $F_{i-1} \in \faces{i-1}(P), F_{i+1} \in \faces{i+1}(P)$ with $F_{i-1} \subseteq F_{i+1}$, there are exactly two $i$-dimensional faces $H \in \faces{i}(P)$ such that $F_{i-1} \subseteq H \subseteq F_{i+1}$. This property is called the \textit{diamond property} of polytopes, and will be widely used in this work.
Therefore, given a flag $F \in \flags(P)$, there is a unique flag $r_i F \in \flags(P)$ which differs from $F$ exactly in its $i$th coordinate. We call the operators $(r_i)_{i=0}^{\dim P-1}$ \textit{flips}, and the group generated by them is called the \textit{monodromy group} of $P$.
The same holds for polyhedra as well, except in dimension $i=0$: in this case, there may be a unique vertex $F_0$ between $F_{-1} = \emptyset$ and the chosen $F_1 \in \faces{1}(P)$. Thus for polyhedra, all the flips $r_1,...,r_{\dim P - 1}$ 
We start this subsection with a claim about the interaction of flips of flags and inclusion.

\begin{claim}\label{cl:ladder}
    Let $P$ be a $d$-dimensional polytope, $F \in \flags(P)$ and let $i,j \in [0,d-1]$ with $i \leq j$. Then $F_i \not\subseteq (r_j ... r_{i+1} r_i F)_j$.
\end{claim}

\begin{proof}
    For any $k \in [i,j]$ denote $F^k = r_k ... r_{i+1} r_i F \in \flags(P)$, see \Cref{fig:flips}. First note that for any $k \in [i,j]$, $F^k_k \neq F_k$, as $F_k = F^{k-1}_k$ and $F^k = r_k F^{k-1}$ and so $F^k$ differs from $F^{k-1}$ in its $k$-th face. Assume towards contradiction that $F_i \subseteq F^j_j$. Denote $k_0 = \max \setdef{k \in [0,j]}{F_k \subseteq F^j_j}$, then $k_0 \geq i$ by hypothesis, and $k_0 < j$ since $F_j \neq F^j_j$ and two distinct $j$-dimensional faces cannot contain one another. However, we claim that $F^{k_0}_{k_0} \subseteq F^j_j$ and $F^{k_0}_{k_0} \subseteq F_{k_0+1}$. Indeed, $F^{k_0}_{k_0} = F^j_{k_0} \subseteq F^j_j$ since $r_{k_0+1},...,r_j$ do not change the $k_0$-dimensional face, and similarly $F^{k_0}_{k_0} \subseteq F^{k_0}_{k_0 + 1} = F_{k_0+1}$ since $r_i,...,r_{k_0}$ do not change the $(k_0+1)$-dimensional face.

    Note that $\supp_P(F_{k_0} \vee F^{k_0}_{k_0}) = F_{k_0 + 1}$, since $F_{k_0}, F^{k_0}_{k_0}$ are distinct faces contained in $F_{k_0+1}$, and $\dim F_{k_0+1} = k_0+1$ while $\dim F_{k_0} = \dim F^{k_0}_{k_0} = k_0$. Therefore, since $F^j_j \supseteq F_{k_0}, F^{k_0}_{k_0}$, we have $F^j_j \supseteq \supp_P(F_{k_0} \vee F^{k_0}_{k_0}) = F_{k_0+1}$, in contradiction to the choice of $k_0$. 
\end{proof}

\begin{figure}\centering
\begin{tikzpicture}[scale=2,line join=bevel]
    \coordinate (F1) at (0,0);
    \coordinate (F2) at (-1,1);
    \coordinate (F2') at (1,1);
    \coordinate (F3) at (0,2);
    \coordinate (F3') at (2,2);
    \coordinate (F4) at (1,3);
    \coordinate (F4') at (3,3);
    \coordinate (F5) at (2,4);
    \coordinate (F5') at (4,4);
    \coordinate (A) at (-0.1,0);
    \coordinate (B) at (-0.2,0);
    \coordinate (C) at (-0.3,0);
    \coordinate (NW) at (-1,1);
    \coordinate (NE) at (1,1);
    \coordinate (OFFSET) at (-1,-1);
    
    \node[anchor=north] at (F1) {$F_{i-1}$};
    \node at (F2) {$F_i$};
    \node[anchor=north west] at (F2') {$(F^i)_i$};
    \node[anchor=west] at (F3) {$F_{i+1}$};
    \node[anchor=west] at (F3') {$\dots$};
    \node[anchor=west] at (F4) {$\dots$};
    \node[anchor=west] at (F4') {$(F^{j-1})_{j-1}$};
    \node[anchor=west] at (F5) {$F_j$};
    \node[anchor=west] at (F5') {$(F^j)_j$};
    
    \draw[red] (C) -- ++(NW) -- ++(NE) -- ++ (NE) -- ++(NE);
    \draw[red] (B) -- ++(NE) -- ++(NW) -- ++ (NE) -- ++(NE);
    \draw[red] (A) -- ++(NE) -- ++(NE) -- ++ (NE) -- ++(NW);
    \draw[red] (F1) -- ++(NE) -- ++(NE) -- ++ (NE) -- ++(NE);
    
    \node at ($(F1)!0.5!(F2)+0.3*(OFFSET)$) {\color{red} $F$};
    \node at ($(F2')!0.5!(F3)+0.25*(OFFSET)$) {\color{red} $F^i$};
    \node at ($(F4')!0.5!(F5)+0.25*(OFFSET)$) {\color{red} $F^{j-1}$};
    \node at ($(F4')!0.5!(F5')-0.1*(NW)$) {\color{red} $F^j$};
    
\end{tikzpicture}\caption{The flags {\color{red} $F, F^i, ..., F^j$}, represented by lines passing through their respective faces.}\label{fig:flips}
\end{figure}
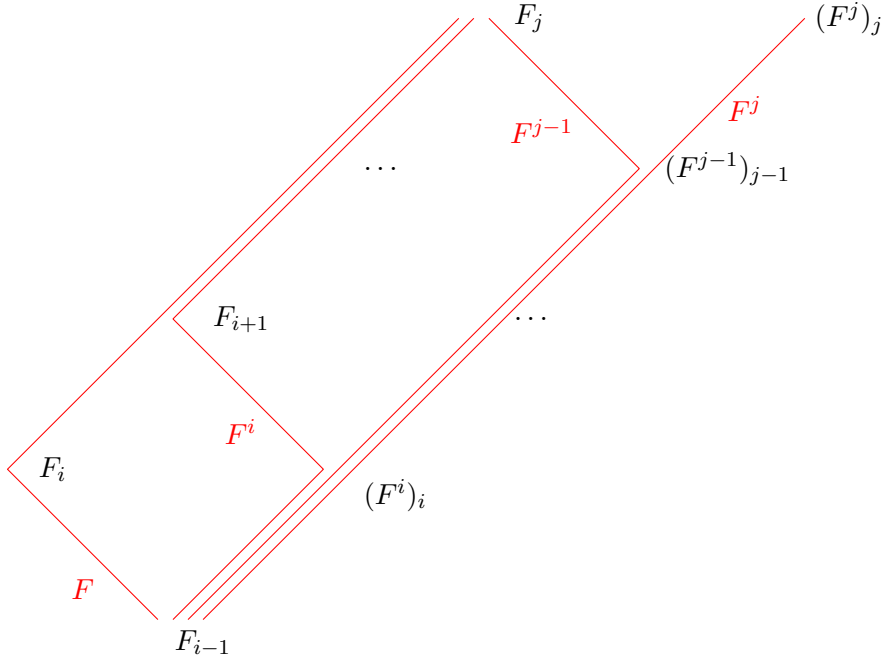

The main lemma of this subsection is as follows.

\begin{lemma}\label{lem:unique_proj}
    Let $P \subseteq \RR^d$ be a $d$-dimensional polytope and let $F \in \flags(P)$. There exists a unique edge $E \in \faces{1}(P)$ that contains $F_0$ and such that $(\supp_{\proj_{E^\perp} P} \proj_{E^\perp} F_i)_{i = 0}^{d-1}$ is a flag of $\proj_{E^\perp} P$. Furthermore,
    \[
        E = (r_1 r_2... r_{d-1} F)_1 .
    \]
\end{lemma}

Saying that the projection of a flag, in a certain direction $n$, results in a flag, is actually a statement about the interaction of the direction $n$ and the normal cones to the faces in the flag, which form a flag themselves (i.e. they are nested). Thus \Cref{lem:unique_proj} is proven by first showing the dual statement:

\begin{lemma}\label{lem:unique_face}
    Let $P \subseteq \RR^d$ be a $d$-dimensional polyhedron with faces of all dimensions, and let $F \in \flags(P)$. There exists a unique facet $G \in \faces{d-1}(P)$ that contains $F_0$ and such that $(F_i \cap G)_{i \in [d]}$ is a flag of $G$. Furthermore,
    \[
        G = (r_{d-1} r_{d-2} ... r_1 F)_{d-1} .
    \]
\end{lemma}

\begin{proof}
    Let $G = (r_{d-1} r_{d-2} ... r_1 F)_{d-1}$, we will show it is the unique facet of $P$ that contains $F_0$ and such that $(F_k \cap G)_{k \in [d]} \in \flags(G)$. For any $i \in [d-1]$ denote $F^{i} = r_i ... r_2 r_1 F \in \flags(P)$, and also denote $F^{0} = F$. In this notation, $G = (F^{d-1})_{d-1}$. Note that we have $(F^{i})_j = (F^{k})_j$ for any $i,j,k \in [0,d-1]$ such that $i,k < j$ or $i,k \geq j$, since $r_l$ only changes the face in dimension $l$, for any $l$.
    First let us show that for any $i \in [0,d-2]$,
    \[
        (F^{i})_i = F_{i+1} \cap G .
    \]
    Indeed, by the above, $(F^{i})_i = (F^{d-1})_i \subseteq (F^{d-1})_{d-1} = G$, and also $(F^{i})_i \subseteq (F^i)_{i+1} = (F^{0})_{i+1} = F_{i+1}$, so $(F^{i})_i \subseteq F_{i+1} \cap G$. But $F_{i+1} \not\subseteq G$ by \Cref{cl:ladder}, which implies that $\dim F_{i+1} \cap G \leq i$, and of course $\dim F^{i}_i = i$. Therefore they are equal. 
    
    Specifically, using $i=0$, we get $F_0 = (F^{0})_0 = F_1 \cap G$ and so $F_0 \subseteq G$. Considering $0 \leq i \leq d-2$, we get that $F_{i+1} \cap G = (F^{i})_i = (F^{d-1})_i$, and since $F_d \cap G = P \cap G = (F^{d-1})_{d-1}$, we see that $(F_i \cap G)_{i \in [d]}$ is just $((F^{d-1})_i)_{i=0}^{d-1}$, which is a flag of $G$ by construction.

    Next, let $G'$ be a facet of $P$ such that $F_0 \subseteq G'$ and assume that $(F_i \cap G')_{i \in [d]}$ is a flag of $G'$. We will show that for any $i \in [0,d]$, $F_i \cap G' = F_i \cap G$ by induction on $i$.
    
    Indeed, for $i=0$ the condition is guaranteed since $F_0 \subseteq G'$ by hypothesis. Assume that $F_i \cap G' = F_i \cap G$, we will show that $F_{i+1} \cap G' = F_{i+1} \cap G$. The face $F_i \cap G'$ is of dimension $i-1$ due to its location in the sequence of faces $(F_j \cap G')_{j \in [d]}$, which is a sequence of faces of dimensions $0,1,2,...,d-1$. Therefore $F_i \not\subseteq G'$, since if $F_i \subseteq G'$ then $\dim F_i \cap G' = \dim F_i = i$, a contradiction. By the same argument, $F_i \not\subseteq G$. Additionally, both $F_{i+1} \cap G'$ and $F_{i+1} \cap G$ are faces of $P$ contained in $F_{i+1}$ and containing $F_i \cap G' = F_i \cap G$. By the diamond property there are exactly two faces of $P$ between $F_i \cap G'$ and $F_{i+1}$. Note that $F_i$ is one of these faces, and that neither $F_{i+1} \cap G'$ nor $F_{i+1} \cap G$ can equal $F_i$ since that would imply $F_i \subseteq G'$ or $F_i \subseteq G$. Thus $F_{i+1} \cap G' = F_{i+1} \cap G$. 

    Using $i=d$ we get $G' = F_d \cap G' = F_d \cap G = G$.
\end{proof}

\begin{proof}[Proof of \Cref{lem:unique_proj}]
    This proof is a standard duality argument: for any polytope $P \subseteq \RR^d$ and $k \in [0,d]$ there is a bijection
    \[
        m_P: \faces{k}(P)_{\geq F_0} \to \faces{d-k}(N_P F_0)
    \]
    given by $L \mapsto N_P L$, where $\faces{k}(P)_{\geq F_0}$ denotes the set of faces of $P$ containing $F_0$ of dimension $k$. The inverse map given by $G \mapsto (F_0+G^\perp) \cap P$. Further, for any $E \in \faces{1}(P)_{\geq F_0}$ and $L \in \faces{k}(P)_{\geq F_0}$, denoting $L' = \proj_{E^\perp} L$ and $P' = \proj_{E^\perp} P$,
    \begin{equation}\label{eq:duality}
        m_{P'}(\supp_{P'} L') \cap E^\perp = m_P(L) \cap E^\perp,
    \end{equation}
    or in other words, $N_{P'} (\supp_{P'} L') \cap E^\perp = N_P L \cap E^\perp$.
    Indeed, for any $\nu \in E^\perp$ we have
    \[
        \nu \in N_P L \Leftrightarrow L \subseteq P(\nu) \Leftrightarrow L' \subseteq P'(\nu) \Leftrightarrow \supp_{P'} L' \subseteq P'(\nu) \Leftrightarrow \nu \in N_{P'} (\supp_{P'} L') .
    \]
    Since $m_{P'}$ is order-reversing, we see that $(\supp_{P'} \proj_{E^\perp} F_i)_{i=0}^{d-1}$ is a flag of $P'$ if and only if $(m_{P'} (\supp_{P'} \proj_{E^\perp} F_{d-i-1}) \cap E^\perp)_{i=0}^{d-1}$ is a flag of $N_P F_0 \cap E^\perp$. Using this fact, \eqref{eq:duality}, and then \Cref{lem:unique_face} with the flag $(m_P(F_{d-k}))_{k=0}^d$, we get existence and uniqueness of the edge $E$. Since this duality is order-reversing, we have $r_k \circ m = m \circ r_{d-k}$, where here $m$ is considered as a map $\setdef{G \in \flags(P)}{G_0 = F_0} \to \flags(N_P F_0)$ acting facewise.
    This proves the 'furthermore' statement.
\end{proof}

\subsection{Some properties of locally anti-blocking polytopes}\label{sec:properties}

In this subsection we recall several useful propositions, and the definition of the graph $\graph{P}$ which encodes useful combinatorial data about the polytope $P$. We start with the following proposition, rephrased in the language of the standard fan $\phist$.

\begin{prop}[{\cite[Proposition 8]{sanyal2023kalai}}]\label{prop:face_cant_cross}
    Let $P$ be a locally anti-blocking polytope in $\RR^d$, $F \in \faces(P)$, and $C,D \in \phist$. If $\relint F \cap \relint C \neq \emptyset$ and $\relint F \cap \relint D \neq \emptyset$ then $N_P F \subseteq \lin (C \cap D)$.
\end{prop}

\begin{remark}\label{rem:generalization}
    \Cref{prop:face_cant_cross} is in fact a characterization of locally anti-blocking polytopes: a polytope is locally anti-blocking if and only if it satisfies the conclusion of the proposition. This condition may be viewed as being ``symmetric w.r.t. $\phist$'', and we can consider the following generalization: for a fan $\Phi$, a polytope $P$ is ``$\Phi$-symmetric'' if for any $F \in \faces(P)$ and $C,D \in \Phi$, if $\relint F \cap \relint C \neq \emptyset$ and $\relint F \cap \relint D \neq \emptyset$, then $N_P F \subseteq \lin (C \cap D)$. 
\end{remark}

We next recall the definition of $\graph{P}$ from \cite{sanyal2023kalai}, see also \cite{reisner1991certain} for the 1-unconditional case.
Let $P$ be a proper locally anti-blocking polytope. Assume that $P$ has the following property: for any $J \subseteq [d]$ with $|J| = 2$ the intersection $P \cap \RR^J$ is either the Cartesian product $[a_1,b_1] \times [a_2,b_2]$ of some two segments in the coordinate axes in $\RR^J$ with $0$ in their relative interior, or it is the convex hull $[a_1,b_1] \vee [a_2,b_2]$ of two such segments. In the former case, we say $P \cap \RR^J$ is \textit{axis-aligned}, and in the latter case, we say it is \textit{a diamond}.
Under this condition, define $\graph{P}$ to be the graph on the vertices $[d]$ such that $\{i,j\}$ is an edge if and only if $P \cap \RR^{\{i,j\}}$ is axis-aligned. As seen in \cite[Lemma 19, Lemma 20]{sanyal2023kalai}, if $P$ minimizes the number of \textit{faces} then these graphs capture many combinatorial properties of $P$, and specifically they determine whether $P$ is a generalized Hanner polytope. We will show an analogous statement for the case where $P$ minimizes the number of flags, see \Cref{cl:hanner_characterization,cor:graph_determines_P}; these proofs are directly analogous to the ones in \cite{sanyal2023kalai}.

Recall that a graph is called a cograph if it is either a single vertex, a complement of a cograph, or a disjoint union of two cographs.
The following two claims will be of use.

\begin{prop}[{\cite[Proposition 12]{sanyal2023kalai}}]\label{prop:graph_calculation}
    Let $P$ and $Q$ be proper locally anti-blocking polytopes such that $\graph{P}, \graph{Q}$ is defined. Then
    \begin{enumerate}
        \item $\graph{P^\circ}$ is the complement graph of $\graph{P}$.
        \item $\graph{P \cap \RR^J}$ is the subgraph of $\graph{P}$ induced by the vertices $J \subseteq [d]$.
        \item $\graph{P \vee Q}$ is the disjoint union $\graph{P} \sqcup \graph{Q}$, where here $P$ and $Q$ are assumed to be embedded in orthogonal subspaces.
    \end{enumerate}
\end{prop}

\begin{lemma}[{\cite[Theorem 2]{corneil1981complement}}]\label{lem:cograph_characterization}
    A graph is a cograph if and only if it does not contain a path of length 3 as an induced subgraph.
\end{lemma}

\section{The inequality part of the main theorem}\label{sec:inequality}

For a proper locally anti-blocking $d$-dimensional polytope $P$, we will show that $|\flags(P)| \geq 2^d \cdot d!$ by induction on the sign. 
We will show that for $C,D \in \phist$ such that $C$ is a facet of $D$, any flag $F \in \signedflags{C}(P)$ induces a distinct flag in $\signedflags{D}(P)$. This will imply that $|\signedflags{D}(P)| \geq \dim D !$ and the result will follow by summing over all of the full-dimensional signs $D \in \phist$. We start by showing how a flag $F \in \signedflags{C}(P)$ induces a flag $G \in \signedflags{D}(P)$.

In short, every face $F_i$ of a flag $F \in \signedflags{C}(P)$ lies in a unique minimal face of $P \cap \lin D$, and furthermore, this minimal face is either of dimension $\dim F_i = i$, or of dimension $i+1$. In the former case we set $G_i$ to be this minimal face, and in the latter we choose $G_i$ to be a suitable $i$-face of this minimal face. This will give a flag $G \in \signedflags{D}(P)$, sitting "above" $F$, see \Cref{fig:raise} for examples. The precise details are given in the next lemma, which is stated for the case $\dim D = d$ for convenience.

\begin{lemma}\label{lem:one_step}
    Let $P$ be a proper locally anti-blocking polytope. Let $D \in \phist$ be full-dimensional and let $C \in \faces{d-1}(D)$ be a facet. Denote by $n$ the unique inner unit normal to $D$ at $C$. Let $F \in \signedflags{C}(P)$. Then there is a flag $G \in \signedflags{D}(P)$ such that for any $k \in [0,d-1]$,
    \begin{enumerate}[label=(\roman*)]
        \item \label{prop:lem1} $G_k \subseteq \aff F_k + \RR_{\geq 0} n$,
        \item \label{prop:lem2}  $n \not\in \lin G_k$, and
        \item \label{prop:lem3} $\supp_{\proj_{\lin C} P} \proj_{\lin C} G_k = F_k$.
    \end{enumerate}
\end{lemma}

\begin{proof}
    For any $k \in [0,d-1]$ set $H_k = \supp_P ((\aff F_k + \RR_{\geq 0} n) \cap P) \in \faces(P)$, and also set $H_{-1} = \emptyset$. Note that by definition of $H_k$, for all $k \in [-1,d-1]$, $F_k \subseteq H_k \subseteq \aff F_k + \RR n$. Therefore $k = \dim F_k \leq \dim H_k \leq \dim F_k + 1 = k+1$. 
    By definition we have the inclusion $H_k \subseteq H_{k+1}$ for all $k$. Note that this inclusion is strict: if for some $k$, $H_k = H_{k+1}$ then $\dim H_k = \dim H_{k+1} = k+1$ and so $\aff H_k = \aff F_k + \RR n$ on the one hand. On the other hand, $\aff H_{k+1} = \aff F_{k+1}$, but $\aff F_{k+1} \neq \aff F_k + \RR n$ since $F_{k+1} \subseteq n^\perp$. This implies that there is some $k_0 \in [0,d-1]$ such that for any $k < k_0$, $\dim H_k = k$ and for any $k \geq k_0$, $\dim H_k = k+1$. 

    Denote $L = \lin C$. We will build the flag $G$ in steps, see \Cref{fig:raise}, while maintaining the three \cref{prop:lem1,prop:lem2,prop:lem3} from the statement of the Lemma. For any $k \in [-1,d]$, in order, we determine $G_k$ as follows. Set $G_{-1} = \emptyset$. \\

    If $k < k_0$, we have $\dim H_k = \dim F_k$, so $H_k = F_k$. In this case $F_k \in \faces(P)$. Set $G_k = F_k$. \Cref{prop:lem1,prop:lem2,prop:lem3} hold trivially since $F_k \subseteq L$, we have $G_k \supseteq G_{k-1}$ by construction, and $\dim G_k = \dim F_k = k$ (these latter two conditions are necessary for $G$ to be a flag). Additionally, since $F_k \subseteq C \subseteq D$, we have $\relint G_k \cap D \neq \emptyset$. \\

    If $k = k_0$, by construction we have $G_{k-1} \subseteq H_{k-1} \subseteq H_k$, and also $\dim H_k = k+1 = \dim G_{k-1} + 2$. By the diamond property, there are exactly two faces of $P$ of dimension $k$ between $G_{k-1}$ and $H_k$.
    Exactly one of these two faces intersects $\relint D$: indeed, if neither intersect $\relint D$ then $H_k \subseteq L + \RR_{\leq 0} n$, which implies that $H_k = F_k$ and in particular $\dim H_k = k$, in contradiction to $k = k_0$; and if both faces intersect $\relint D$ then $H_k \setminus G_{k-1} \subseteq \relint D = L + \RR_{>0} n$, which is a contradiction to the fact that $F_k \subseteq H_k \cap L$. Denote by $G_k$ the unique $k$-dimensional face of $P$ between $G_{k-1}$ and $H_k$ that intersects $\relint D$. Note that \cref{prop:lem2} holds in this case: otherwise, $n \in \lin G_k$, and since by construction $n \not\in \lin G_{k-1}$ and $G_k \supseteq G_{k-1}$, we get $\aff G_k = \aff G_{k-1} + \RR n$. Since $G_k$ intersects $L + \RR_{>0}n$, we see that $P \cap (\aff F_{k-1} + \RR_{>0}n) \neq \emptyset$ and so by definition $H_{k-1} \supsetneq F_{k-1}$, in contradiction. Since $G_k \cap \relint D \neq \emptyset$ and $D$ is full-dimensional, we also have $\relint G_k \cap \relint D \neq \emptyset$ and in particular $\relint G_k \cap D \neq \emptyset$. \Cref{prop:lem1,prop:lem3} will be verified shortly. \\

    If $k > k_0$, again by construction $G_{k-1} \subseteq H_{k-1} \subseteq H_k$ and $\dim H_k = \dim G_{k-1} + 2$. By the diamond property there are exactly two $k$-dimensional faces of $P$ between $G_{k-1}$ and $H_k$. Exactly one of these faces $S$ has the property $n \not\in \lin S$: indeed, since $n \not\in \lin G_{k-1}$, if both faces have $n$ in their linear span then they have equal affine span, which is impossible; however, consider $H_{k-1}$: we have already established $G_{k-1} \subsetneq H_{k-1} \subsetneq H_k$, and since $k-1 \geq k_0$ we have $\dim H_{k-1} = k$. Since by definition, $H_{k-1} \subseteq \aff F_{k-1} + \RR n$, we see that $\lin H_{k-1} = \lin F_{k-1} + \RR n$ and specifically $n \in \lin H_{k-1}$. Thus at least one of the faces strictly between $G_{k-1}$ and $H_k$ has $n$ in their linear span. Define $G_k$ to be the unique face strictly between $G_{k-1}$ and $H_k$ with $n \not\in \lin G_k$, i.e. \cref{prop:lem2} holds.    
    Since by construction $G_{k-1}$ itself intersects $\relint D = L + \RR_{>0} n$ and $G_{k-1} \subseteq G_k$, $G_k$ also intersects $\relint D$. By the same argument as above, $\relint G_k \cap D \neq \emptyset$. See \Cref{fig:FGH}. \\
    
    We turn to verifying \cref{prop:lem1,prop:lem3} in the cases $k \geq k_0$.
    Since we established that $H_k \subseteq \aff F_k + \RR n$ and $G_k$ is a face of $H_k$, $G_k \subseteq \aff F_k + \RR n$. Assume $G_k \not\subseteq \aff F_k + \RR_{\geq 0} n$, then $\relint G_k \cap (L + \RR_{<0}n) \neq \emptyset$, and by construction $G_k \cap (L + \RR_{>0}n) \neq \emptyset$ which also implies $\relint G_k \cap (L + \RR_{>0}n) \neq \emptyset$. By \Cref{prop:face_cant_cross}, $N_P G_k \subseteq L = n^\perp$ and so $n \in \lin G_k$, in contradiction. This proves \cref{prop:lem1}.
    \Cref{prop:lem3} quickly follows: since $G_k \subseteq \aff F_k + \RR_{\geq 0} n$, $\proj_L G_k \subseteq \aff F_k$, and so $\supp_{\proj_L P} \proj_L G_k \subseteq F_k$. Since $n \not\in \lin G_k$, $\dim \proj_L G_k = k$, and so $\dim \supp_{\proj_L P} \proj_L G_k \geq k$.  Thus $\supp_{\proj_L P} \proj_L G_k = F_k$.

    Finally, set $G_d = P$. Since $P$ is \textit{proper} locally anti-blocking, $\relint P \cap D \neq \emptyset$. We have built a flag $G$ that satisfies \cref{prop:lem1,prop:lem2,prop:lem3} and such that for all $k \in [0,d]$, $\relint G_k \cap D \neq \emptyset$. Since $D$ is full-dimensional, this implies $\sign G = D$ and we are done.
\end{proof}

\begin{figure}
    \centering
    \begin{tikzpicture}[scale=2,line join=bevel]
        \coordinate (A1) at (1,0);
        \coordinate (A2) at (-1,0);
        \coordinate (B1) at (0,1);
        \coordinate (B2) at (0,-1);
        \coordinate (C1) at (-1,1);
        \coordinate (C2) at (1,-1);
        
        \coordinate (O) at (0,0);
        \coordinate (E1) at (1.5,0);
        \coordinate (E2) at (0,1.5);
        
        \draw (A1) -- (B1) -- (C1) -- (A2) -- (B2) -- (C2) -- cycle;

        \draw[green,line width=3] (A1) -- (A2);
        
        \draw[blue,line width=3] (A1) -- (B1);

        \node at (A1) [draw, circle, minimum size=0.15cm, inner sep=0pt, path picture={
            \clip (path picture bounding box.center) circle (0.15cm);
            \fill[green] (-0.2,0) rectangle (0.2,0.2);
            \fill[blue]  (-0.2,-0.2) rectangle (0.2,0);
        }] {};
        
        \draw[->] (O) -- ++(E1);
        \draw[->] (O) -- ++(E2);
        \node[anchor=south] at (E1) {$x$};
        \node[anchor=south west] at (E2) {$y$};
        \node[anchor=north west] at (A1) {$F_0 = H_0 = G_0$};
        \node[anchor=north] at (O) {$F_1$};
        \node[anchor=south west] at ($(A1)!0.5!(B1)$) {$G_1$};
        \node at ($(A1)!0.5!(B2)$) {$H_1$};
    \end{tikzpicture}
        
    \caption{A flag $F \in \protect\signedflags{C}(P)$ (in green), and the $0,1$ faces of $\chi_C^D(F) \in \protect\signedflags{D}(P)$, in blue, where $C$ is spanned by $x$ and $D$ by $x,y$. The face $H_1$ equals $P$ itself here.}
    \label{fig:FGH}
\end{figure}
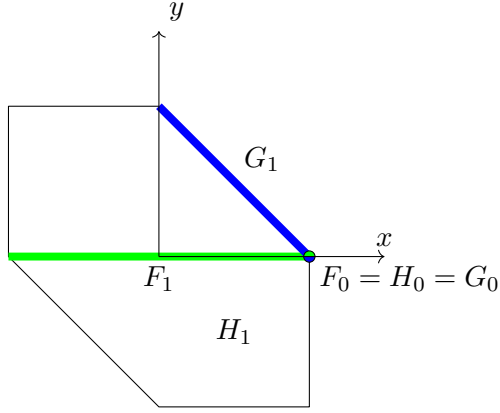

\begin{prop}\label{prop:count}
    Let $P \subseteq \RR^d$ be a proper locally anti-blocking polytope and let $D \in \phist$ be full-dimensional. There are at least $d!$ flags of $P$ with sign $D$.
\end{prop}

This proposition is an immediate corollary of the following lemma:

\begin{lemma}\label{lem:construction_of_D}
    Let $P \subseteq \RR^d$ be a proper locally anti-blocking polytope. There exist injections $\chi_C^D: \signedflags{C}(P) \to \signedflags{D}(P)$ for all $D \in \phist$ and a facet $C \in \faces{\dim D - 1}(D)$, such that the following properties hold.
    \begin{enumerate}[label=(\roman*)]
        \item \label{prop:1} For any $D \in \phist$ and $C \neq C' \in \faces{\dim D -1}(D)$, $\Ima \chi_C^D \cap \Ima \chi_{C'}^D = \emptyset$.
        \item \label{prop:2} For any $D \in \phist$, $|\signedflags{D}(P)| \geq \dim D!$.
    \end{enumerate}
\end{lemma}

\begin{proof}
    The proof is by induction on the sign $D$. For the minimal sign $\{0\} \in \phist$ there is nothing to prove.

    Let $D \in \phist$ and set $n = \dim D$. Assume by induction that for any $C \in \faces{n-1}(D)$, $|\signedflags{C}(P)| \geq (n - 1)!$.
    For any $C \in \faces{n-1}(D)$ construct the map $\chi_{C}^D: \signedflags{C}(P) \to \signedflags{D}(P)$ that maps $F \in \signedflags{C}(P)$ to the $G \in \signedflags{D}(P)$ given by \Cref{lem:one_step}, when applied to $P \cap \lin D$, in which $D$ is full-dimensional.
    The lemma also guarantees that for any $k \in [0,n-1]$, $G_k \subseteq \aff F_k + \RR_{\geq 0} n_C^D$, where $n_C^D$ is the unique unit normal to $\lin C$ such that $D \subseteq C + \RR_{\geq 0} n_C^D$; that $n_C^D \not\in \lin G_k$, and that $\supp_{\proj_{\lin C} P} \proj_{\lin C} G_k = F_k$. 

    We first claim that the maps $\chi_C^D$ are injections. Indeed, assume towards contradiction that $F \neq F'$ are distinct flags in $\signedflags{C}(P)$ for some $C \in \faces{n-1}(D)$, and that $\chi_C^D F = \chi_C^D F' = G$. There is some $k \in [0,n-1]$ such that $F_k \neq F_k'$. Since
    \[
        G_k \subseteq (\aff F_k + \RR_{\geq 0} n_C^D) \cap (\aff F_k' + \RR_{\geq 0} n_C^D) = \aff (F_k \cap F_k') + \RR_{\geq 0} n_C^D
    \]
    and $n_C^D \not\in \lin G_k$, we find that $G_k$ is a subset of $\aff (F_k \cap F_k') + \RR_{\geq 0} n_C^D$ of codimension at least 1. But $\dim \aff (F_k \cap F_k') < k$, which implies that $\dim G_k < k$, in contradiction. 

    We turn to showing \cref{prop:1,prop:2}. For the remainder of this proof, consider $P \cap D$ as a full-dimensional polytope in $\lin D$, and specifically, whenever we take the orthogonal complement $v^\perp$ of a vector $v \in \lin D$, we mean the orthogonal complement in $\lin D$.
    In order to show \cref{prop:1}, assume towards contradiction that for $C \neq C' \in \faces{n-1}(D)$, there are $F \in \signedflags{C}(P)$ and $F' \in \signedflags{C'}(P)$ such that $\chi_C^D(F) = \chi_{C'}^D(F') = G \in \signedflags{D}(P)$. Since for any $k \in [0,n-1]$,
    \begin{equation}\label{eq:2}
        \supp_{\proj_{\lin C} P} \proj_{\lin C} G_k = F_k \qquad{and}\qquad \supp_{\proj_{\lin C'} P} \proj_{\lin C'} G_k = F_k' ,
    \end{equation}
    specifically for $k=0$, we see that $\supp_{\proj_{\lin C} P} \proj_{\lin C} G_0 = F_0$ is a vertex of $\proj_{\lin C} P$, which implies that $N_{\proj_{\lin C} P} F_0 \cap \lin C$ is $(n-1)$-dimensional and thus $E := (G_0 + \RR n_C^D) \cap P = (F_0 + \RR n_C^D) \cap P$ is an edge of $P \cap \lin D$ containing $G_0$. Similarly, $E' := (G_0 + \RR n_{C'}^D) \cap P$ is an edge of $P \cap \lin D$ containing $G_0$.
    Using \eqref{eq:2} again, we see that $(\supp_{\proj_{E^\perp} (P \cap \lin D)} \proj_{E^\perp} G_k)_{k=0}^{n-1} = (F_k)_{k=0}^{n-1}$ is a flag of $\proj_{E^\perp} (P \cap \lin D)$, and similarly $(\supp_{\proj_{(E')^\perp} (P \cap \lin D)} \proj_{(E')^\perp} G_k)_{k=0}^{n-1} = (F_k')_{k=0}^{n-1}$ is a flag of $\proj_{(E')^\perp} (P \cap \lin D)$. Since $E \neq E'$, we find two distinct edges of $P \cap \lin D$ that satisfy the condition of \Cref{lem:unique_proj}, in contradiction.

    Finally, for \cref{prop:2}, since by induction $|\signedflags{C}(P)| \geq (n - 1)!$ for all $C \in \faces{n-1}(D)$, and since by \cref{prop:1}, $\signedflags{D}(P)$ contains the disjoint union $\bigsqcup_{C \in \faces{n-1}(D)} \chi_C^D \parens*{\signedflags{C}(P)}$, we have
    \[
        |\signedflags{D}(P)| \geq \sum_{C \in \faces{n-1}(D)} |\signedflags{C}(P)| = n \cdot (n - 1)! = n ! = \dim D ! .
    \]
\end{proof}

The inequality part of \Cref{thm:main_equiv} is a simple corollary of \Cref{prop:count} by summing over all the full-dimensional signs $D \in (\phist)_d$. 

\begin{remark}
    The inequality part still holds for general fans $\Phi$. More precisely: under the condition that for any $D \in \Phi$ and $C$ a facet of $D$ with inner normal $n$, $C + \RR_{\geq 0} n \subseteq D$, and that $P$ is $\Phi$-symmetric (as defined in \Cref{rem:generalization}), the same proof holds verbatim and shows that $|\flags(P)| \geq |\flags(\Phi)|$, i.e. $P$ has at least as many flags as the fan $\Phi$. 
\end{remark}

\section{Equality case}\label{sec:equality}

We follow the strategy of Sanyal and Winter in \cite{sanyal2023kalai}. Namely, if $P$ is a minimizer for the number of flags, we will show that for any coordinate subspace $H$, $P \cap H$ is also a minimizer. This allows us to define the graph $\graph{P}$ of $P$ which encodes some of the combinatorial information of $P$, which was recalled in \Cref{sec:properties}. We will see that since all $P \cap H$, and specifically those with $\dim H = 4$, will be shown to be minimizers, the graph $\graph{P}$ is a co-graph, and therefore $P$ is a Hanner polytope.

\begin{prop}\label{prop:min_downwards}
    Let $P$ be a normalized locally anti-blocking $d$-dimensional polytope and assume $|\flags(P)| = 2^d \cdot d!$. Then for any coordinate subspace $H$ of $\RR^d$, $|\flags(P \cap H)| = 2^{\dim H} \cdot \dim H !$.
\end{prop}

\begin{proof}
    The proof uses \Cref{lem:construction_of_D} to show that if a certain coordinate section does not minimize flags, then any coordinate intersection sitting "above" it must also not minimize, and specifically the full polytope doesn't.
    
    Assume there exists some coordinate subspace $H$ with $|\flags(P \cap H)| > 2^{\dim H} \cdot \dim H!$. Since $|\flags(P \cap \RR^d)| = |\flags(P)| = 2^d \cdot d!$ by hypothesis, there exists two coordinate subspaces $H_0 \subseteq H_1$ with $n := \dim H_1 = \dim H_0 + 1$ such that $|\flags(P \cap H_0)| > 2^{n-1} \cdot (n-1)!$ and $|\flags(P \cap H_1)| = 2^n \cdot n!$. 
    Therefore by the pigeonhole principle there is some $D_0 \in \phist$ with $\lin D_0 = H_0$ such that $\left| \signedflags{D_0}(P \cap H_0) \right| > (n-1)!$. Set some $D_1 \in \phist$ such that $\lin D_1 = H_1$ and $D_0 \in \faces{n-1}(D_1)$. 
    By \Cref{lem:construction_of_D} we get 
    \begin{align*}
        \left| \flags(P \cap H_1) \right| =& \sum_{\substack{D \in \phist \\ \lin D = H_1}} |\signedflags{D}(P \cap H_1)| = |\signedflags{D_1}(P \cap H_1)| + \sum_{\substack{D_1 \neq D \in \phist \\ \lin D = H_1}} |\signedflags{D}(P \cap H_1)| \\
        \geq& |\signedflags{D_1}(P \cap H_1)| + \sum_{\substack{D_1 \neq D \in \phist \\ \lin D = H_1}} \dim D ! \\
        \geq& \left| \Ima \chi_{D_0}^{D_1} \right| + \sum_{D_0 \neq C \in \faces{n-1}(D_1)} \left| \Ima \chi_C^{D_1} \right| + (2^n-1) \cdot n! \\
        =& |\signedflags{D_0}(P \cap H_0)| + \sum_{D_0 \neq C \in \faces{n-1}(D_1)} \left| \signedflags{C}(P \cap H_0) \right| + (2^n - 1) \cdot n! \\
        >& (n-1)! + (n-1) \cdot (n-1)! + (2^n - 1) \cdot n! = 2^n \cdot n! , 
    \end{align*}
    in contradiction to the choice of $H_1$.
\end{proof}

This implies that any 2-dimensional coordinate section must be either $\square^2$ or $\lozenge^2$.

\begin{claim}\label{cl:have_graph}
    Let $P$ be a normalized locally anti-blocking $d$-dimensional polytope with $|\flags(P)| = 2^d \cdot d!$. Then for any coordinate subspace $H$ of $\RR^d$ with $\dim H = 2$, either $P \cap H = \square^d \cap H$, or $P \cap H = \lozenge^d \cap H$.
\end{claim}

\begin{proof}
    Let $H$ be a 2-dimensional coordinate subspace. By \Cref{prop:min_downwards} we have $|\flags(P \cap H)| = 8$. Note that any 2-dimensional $n$-gon has exactly $n$ vertices, each contained in exactly 2 edges, and therefore has exactly $2n$ flags. Thus $P \cap H$ is a quadrilateral. Since it is also locally anti-blocking it must be either $\square^d \cap H$ or $\lozenge^d \cap H$. 
\end{proof}

By \Cref{cl:have_graph}, for any $P$ which minimizes the flag number we may define the graph $\graph{P}$ as in \Cref{sec:preq}.
Our goal now is to present an analogue of \cite[Lemma 19]{sanyal2023kalai}; that is, to show that if $P$ minimizes flags, then it can be recovered from $\graph{P}$. 

Recall the definition of $\phist$. Any cone $C \in \phist$ can be uniquely presented as the positive span of some subset $J_C \subseteq \{\pm e_i\}_{i \in [d]}$. Denote by $1_C \in C$ the vector $1_C = \sum_{v \in J_C} v$, i.e. for full-dimensional $C \in \phist$ the vectors $1_C$ are exactly the vertices of the unit cube $\square^d$, and for general $C \in \phist$ the vectors $1_C$ are always barycenters of some face of $\square^d$.

We prove the following proposition, which says that the condition $1_D \in P$, for $D \in \phist$, is only deteremined by the condition $\forall C \in \faces{2}(D): 1_C \in P$. 

\begin{prop}\label{prop:clique_graph_condition}
    Let $P$ be a normalized locally anti-blocking polytope in $\RR^d$ such that $|\flags(P)| = 2^d \cdot d!$ and let $D \in \phist$ of dimension $\geq 2$. Assume that for any $C \in \faces{2}(D)$, $1_C \in P$. Then $1_D \in P$. 
\end{prop}

\begin{proof}
    The proof is by induction on $k := \dim D$. For $k = 2$ there is nothing to prove.
    
    Let $k \geq 3$ and assume $1_D \not\in P$ towards contradiction. By the induction hypothesis, for any facet $C \in \faces{k-1}(D)$, $1_C \in P$. Pick some $v \in J_D$ (i.e. $v \in D$ is a standard basis vector, up to sign) and denote $C_0 = D \cap v^\perp \in \faces{k-1}(D)$. Note that for any $C_0 \neq C \in \faces{k-1}(D)$, $1_C \in P$; also $v \in P$ as previously stated. Thus, since $P$ is normalized, $P(v)$ is $(k-1)$-dimensional. 
    Let $F \in \signedflags{D}(P)$ be a flag with $F_{k-1} = P(v)$. For any $v \neq u \in J_D$, $u \in \lin P(v)$, and so for any $C_0 \neq C \in \faces{k-1}(D)$, $F \not\in \Ima \chi_C^D$ by \cref{prop:lem2} of \Cref{lem:one_step}. However, $F_0 \neq 1_D$ and $\innerp{F_0}{v} = 1$ (since $F_0 \subseteq F_{k-1}$), and so $\proj_{C_0} F_0 \neq 1_{C_0}$ is not a vertex of $P \cap C_0$. Therefore $F \not\in \Ima \chi_{C_0}^D$ as well, by \cref{prop:lem3} of \Cref{lem:one_step}. In all, for any $C \in \faces{k-1}(D)$, $F \not\in \Ima \chi_C^D$, in contradiction to minimality. Thus $1_D \in P$. 
\end{proof}

As a corollary, we see that for normalized locally anti-blocking polytopes $P$ with the minimal number of flags, $P$ is determined by $\graph{P}$. 

\begin{corollary}\label{cor:graph_determines_P}
    Let $P$ be a normalized locally anti-blocking polytope in $\RR^d$ and assume $|\flags(P)| = 2^d \cdot d!$. Then 
    \[
        P = \calc(\graph{P}) := \bigvee_D 1_D ,
    \]
    where the convex hull runs over all $D \in \phist$ such that $D \subseteq \RR^J$ for some $J \subseteq [d]$ which induces a clique in $\graph{P}$. 
\end{corollary}

\begin{proof}
    Let $D \in \phist$ with $D \subseteq \RR^J$ for some $J \subseteq [d]$ which induces a clique in $\graph{P}$. By definition of $\graph{P}$, for any $C \in \faces{2}(D)$, $1_C \in P$. By \Cref{prop:clique_graph_condition}, $1_D \in P$. Thus the $\supseteq$ inclusion holds.
    Let $v \in \faces{0}(P)$ and set $J \subseteq [d]$ to be the set of nonzero coordinates of $v$. If $|J| = 1$, $v = \pm e_i$ for some $i \in [d]$ and must be in $\calc(\graph{P})$. 
    If $|J| \geq 2$, since $P$ is locally anti-blocking, for any coordinate 2-plane $H \subseteq \RR^J$, $\proj_H v \in P$ is not contained in any of the coordinate axes. Therefore by \Cref{cl:have_graph}, $P \cap H = \square^d \cap H$ is axis-aligned and so the coordinates corresponding to $H$ form an edge in $\graph{P}$. Therefore $J$ induces a clique in $\graph{P}$, meaning $v \in \calc(\graph{P})$ and the $\subseteq$ inclusion also holds.
\end{proof}

\begin{remark}
    As another corollary, we see that any normalized locally anti-blocking polytope that minimizes flags has vertices in $\{\pm 1, 0\}^d$. 
\end{remark}

We also need the following direct analogue of \cite[Lemma 20]{sanyal2023kalai}:

\begin{claim}\label{cl:hanner_characterization}
    Let $P$ be a normalized locally anti-blocking polytope in $\RR^d$ and assume $|\flags(P)| = 2^d \cdot d!$. Then $P$ is a Hanner polytope if and only if $\graph{P}$ is a cograph.
\end{claim}

\begin{proof}
    Consider the map
    \[
        \mathcal{G}: \{\text{Normalized, flag-minimizing locally anti-blocking polytopes}\} \to \{\text{Graphs}\}
    \]
    given by $P \mapsto \graph{P}$. By \Cref{cor:graph_determines_P}, $\mathcal{G}$ is 1-to-1. Therefore it is enough to check that $\mathcal{G}$ maps the subset of Hanner polytopes onto the subset of cographs. By the recursive definition of Hanner polytopes, it is enough to check 3 cases:
    \begin{itemize}
        \item If $d = 1$, $\graph{P}$ is a single-vertex graph, and so it is a cograph.
        \item If $P = Q^\circ$ for some previously constructed Hanner polytope $Q$, then $\graph{Q}$ is a cograph by induction, and so $\graph{P}$ is the complement graph of $\graph{Q}$ by \Cref{prop:graph_calculation} and therefore is also a cograph.
        \item If $P = Q_1 \vee Q_2$ for some proviously constructed Hanner polytopes $Q_1,Q_2$ lying in orthogonal subspaces, then by induction $\graph{Q_1},\graph{Q_2}$ are cographs. By \Cref{prop:graph_calculation}, $\graph{P} = \graph{Q_1} \sqcup \graph{Q_2}$ is also a cograph.
    \end{itemize}
    By definition, all cographs can be constructed this way from a Hanner polytope.
\end{proof}

\begin{claim}\label{cl:is_cograph}
    Let $P$ be a normalized locally anti-blocking $d$-dimensional polytope with $|\flags(P)| = 2^d \cdot d!$ and let $J \subseteq [d]$ with $|J| = 4$. Then $\graph{P}[J]$, the induced graph on the vertex set $J$, is not a path of length 3.
\end{claim}

\begin{proof}
    Without loss of generality, $J = \{1,2,3,4\}$, and set $H = \RR^J$ to be the coordinate subspace spanned by the corresponding standard basis vectors. Assume towards contradiction that $\graph{P}[J]$ is the graph $\Pi_3$ whose edges are $\{(1,2),(2,3),(3,4)\}$. By \Cref{cor:graph_determines_P,prop:graph_calculation}, $P \cap H = \calc(\graph{P}[J]) = \calc(\Pi_3)$. A calculation (carried out in \ref{appendix}) gives $|\flags(\calc(\Pi_3))| = 448 > 384 = 2^4 \cdot 4!$, which is a contradiction to \Cref{prop:min_downwards}.
\end{proof}

We have now gathered all the components to prove the equality case in \Cref{thm:main_equiv}.

\begin{proof}[Proof of the equality part of \Cref{thm:main_equiv}]
    Assume $|\flags(P)| = 2^d \cdot d!$. By \Cref{cl:have_graph}, the graph $\graph{P}$ is defined, and by \Cref{cl:is_cograph} and the characterization \Cref{lem:cograph_characterization} it is a cograph. By \Cref{cl:hanner_characterization} we get that $P$ is a Hanner polytope. 
\end{proof}

\appendix
\gdef\thesection{Appendix \Alph{section}}
\section{Flag number calculation for $\calc(\Pi_3)$}\label{appendix}

By definition, $\calc(\Pi_3)$ has vertices
\begin{equation}\label{eq:vectors}
    \begin{pmatrix}
        \pm 1 \\ \pm 1 \\ 0 \\ 0
    \end{pmatrix} , 
    \begin{pmatrix}
        0 \\ \pm 1 \\ \pm 1 \\ 0
    \end{pmatrix} ,
    \begin{pmatrix}
        0 \\ 0 \\ \pm 1 \\ \pm 1
    \end{pmatrix} .
\end{equation}

By \Cref{prop:graph_calculation}, $\calc(\Pi_3)^\circ$ has vertices
\begin{equation*}\label{eq:vectors_dual}
    \begin{pmatrix}
        \pm 1 \\ 0 \\ 0 \\ \pm 1
    \end{pmatrix}, 
    \begin{pmatrix}
        0 \\ \pm 1 \\ 0 \\ \pm 1
    \end{pmatrix} ,
    \begin{pmatrix}
        \pm 1 \\ 0 \\ \pm 1 \\ 0
    \end{pmatrix} .
\end{equation*}
For any vertex $v$ of \eqref{eq:vectors}, its vertex figure has the same combinatorial structure as the dual facet of $\calc(\Pi_3)^\circ$ to $v$, see \cite[Section 2]{ziegler1993lectures}. Therefore the number of flags containing $v$ is the number of flags of the dual facet. For vertices in the first and third columns of \eqref{eq:vectors}, the dual facet is an affine image of the polytope with vertices
\[
    \begin{pmatrix}
        \pm 1 \\ 0 \\ \pm 1
    \end{pmatrix} ,
    \begin{pmatrix}
        1 \\ \pm 1 \\ 0
    \end{pmatrix} ,
\]
see \Cref{fig:1}.

Note that every edge of a 3-dimensional polytope contains exactly 2 vertices and is contained in exactly 2 facets. Therefore there are exactly 4 flags that contain every edge, and so the above polytope has exactly 44 flags. For vertices of \eqref{eq:vectors} of the second column, the dual facet is an affine image of the polytope with vertices
\[
    \begin{pmatrix}
        0 \\ 1 \\ \pm 1
    \end{pmatrix},
    \begin{pmatrix}
        \pm 1 \\ -1 \\ 0
    \end{pmatrix} ,
\]
which is simply a 3-dimensional simplex, and so has 24 flags. In all, there are $8 \cdot 44 + 4 \cdot 24 = 448$ flags in $\calc(\Pi_3)$.

\begin{figure}\centering
\begin{tikzpicture}[scale=2,line join=bevel]
    \coordinate (A1) at (1,0,1);
    \coordinate (A3) at (1,0,-1);
    \coordinate (A2) at (1,1,0);
    \coordinate (A4) at (1,-1,0);
    \coordinate (B1) at (-1,0,1);
    \coordinate (B2) at (-1,0,-1);
    
    \draw (A1) -- (A2) -- (A3) -- (A4) -- cycle;
    \draw (A1) -- (A2) -- (B1) -- cycle;
    \draw (A2) -- (A3) -- (B2) -- cycle;
    \draw (A3) -- (A4) -- (B2) -- cycle;
    \draw (A4) -- (A1) -- (B1) -- cycle;
    \draw (B1) -- (B2) -- (A2) -- cycle;
    \draw (B1) -- (B2) -- (A4) -- cycle;
    \draw [fill opacity=0.7,fill=green!80!black] (A1) -- (A2) -- (A3) -- (A4) -- cycle;
    \draw [fill opacity=0.7,fill=orange!80!black] (A4) -- (A1) -- (B1) -- cycle;
    \draw [fill opacity=0.7,fill=blue!80!black] (A1) -- (A2) -- (B1) -- cycle;
    \draw [fill opacity=0.7,fill=purple!80!black] (B1) -- (B2) -- (A2) -- cycle;
\end{tikzpicture}\caption{The dual facet for vertices in the first and third columns.}\label{fig:1}
\end{figure}
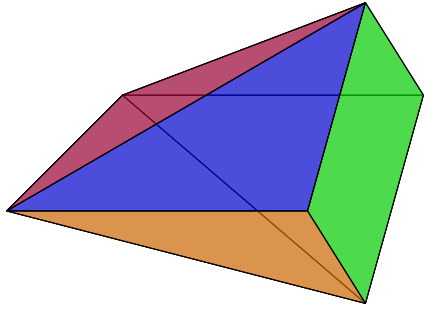

\bibliographystyle{plain}
\bibliography{refs.bib}

{\small
\noindent School of Mathematical Sciences, Tel Aviv University, Ramat
Aviv, Tel Aviv, 69978, Israel.\vskip 2pt
\noindent Email: arnonchor@gmail.com.
}

\end{document}